\documentclass[11pt, a4paper]{amsart}
\usepackage{amsthm}
\usepackage[margin=2.1cm]{geometry}
\usepackage[colorlinks=true,citecolor=cyan,backref=page]{hyperref}
\usepackage{amsmath}
\usepackage{amsfonts}
\usepackage{amssymb}
\usepackage{stmaryrd} 
\usepackage{amssymb} 
\usepackage{mathrsfs} 
\usepackage{esint} 
\usepackage{array} 
\usepackage{latexsym}
\usepackage{lmodern} 
\usepackage{graphicx} 
\usepackage{graphics} 
\usepackage{epsfig} 
\usepackage{color} 
\usepackage[dvipsnames]{xcolor}
\usepackage[all]{xy}
\usepackage[new]{old-arrows}
\usepackage{textcomp}
\usepackage{stmaryrd}
\DeclareGraphicsExtensions{.eps,.pdf,.jpeg,.png} 
\usepackage{fancyhdr}
\usepackage{multirow} 
\usepackage{comment}
\usepackage{pgf,tikz}
\usepackage{paralist}
\usepackage{faktor}
\usepackage{dirtytalk}
\usepackage{url}
\usepackage{blkarray,bigstrut}
\usepackage{nicematrix}
\usepackage{enumerate}

\makeatletter
\renewcommand*\env@matrix[1][*\c@MaxMatrixCols c]{%
  \hskip -\arraycolsep
  \let\@ifnextchar\new@ifnextchar
  \array{#1}}
\makeatother

\newtheorem{corollary}{Corollary}[section]
\newtheorem{proposition}{Proposition}[section]
\newtheorem{theorem}{Theorem}[section]
\newtheorem{lemma}{Lemma}[section]

\theoremstyle{definition}
\newtheorem{definition}{Definition}[section]
\newtheorem{remark}{Remark}[section]
\newtheorem{example}{Example}[section]

\address[Nathan Chapelier-Laget]{Universit\'e du Qu\'ebec \`a Montr\'eal\\
LaCIM et D\'epartement de Math\'ematiques\\ CP 8888 Succ. Centre-Ville\\
Montr\'eal, Qu\'ebec, H3C 3P8\\ Canada}
\email{nathan.chapelier@gmail.com}
\urladdr{https://www.nathanchapelier.fr/home}

\title[A symmetric group action on $H^0(\widehat{X}_{W(\widetilde{A}_n)})$]
 {A symmetric group action on the irreducible components of the Shi variety associated to $W(\widetilde{A}_n)$}

\author{Nathan Chapelier-Laget}

\begin{document}

\maketitle

\begin{abstract}
Let $W_a$ be an affine Weyl group with corresponding finite root system $\Phi$. In \cite{JYS1} Jian-Yi Shi characterized each element $w \in W_a$ by a $ \Phi^+$-tuple of integers $(k(w,\alpha))_{\alpha \in \Phi^+}$ subject to certain conditions. In \cite{NC1} a new interpretation of the coefficients $k(w,\alpha)$ is given. This description led us to define an affine variety $\widehat{X}_{W_a}$, called the Shi variety of $W_a$, whose integral points are in bijection with $W_a$.  It turns out that this variety has more than one irreducible component, and the set of these components, denoted $H^0(\widehat{X}_{W_a})$, admits many interesting properties. In particular the group $W_a$ acts on it. In this article we show that the set of irreducible components of  $\widehat{X}_{W(\widetilde{A}_n)}$ is in bijection with the conjugacy class of $(1~2~\cdots~n+1) \in W(A_n) = S_{n+1}$. We also compute the action of $W(A_n)$ on $H^0(\widehat{X}_{W(\widetilde{A}_n)})$.
\end{abstract}

{
  \hypersetup{linkcolor=blue}
  \tableofcontents
}

\section{Introduction}

\subsection{General background on Weyl groups}

Let $V$ be a Euclidean space with inner product $( -, -)$ and denote $||x|| = \sqrt{(x,x)}$. Let $\Phi$ be an irreducible crystallographic root system in $V$ with simple system $\Delta =  \{\alpha_1,\dots, \alpha_n\}$. Let $m=|\Phi^+|$.  From now on, when we will say \say{root system} it will always mean irreducible crystallographic root system. 

\medskip

Let $W$ be the \emph{Weyl group} associated to $\mathbb{Z}\Phi$, that is the maximal (for inclusion) reflection subgroup of the orthogonal group  $O(V)$ admitting $\mathbb{Z}\Phi$ as a $W$-equivariant lattice.  For $\alpha \in \Phi$ we denote by $s_{\alpha}$ the linear reflection of $V$ defined as follows:
$$
\begin{array}{ccccc}
s_{\alpha}  & : & V & \longrightarrow & V \\
                 &   & x & \longmapsto     & x-2\frac{( \alpha, x )}{( \alpha, \alpha )}\alpha.
\end{array}
$$ 

We denote $s_i := s_{\alpha_i}$ for $i=1,\dots, n$ and $S = \{s_{\alpha}~|~\alpha \in \Delta\}$ so that $(W,S)$ is a Coxeter system of rank $n$ (see \cite{Brenti2005}, \cite{Bour68} or \cite{Kane2001} for good references on the subjet).

We identify $\mathbb{Z}\Phi$ and the group of its associated translations and we denote by $\tau_x$ the translation corresponding to $x \in \mathbb{Z}\Phi$. 

\newpage

Let $k \in \mathbb{Z}$ and $\alpha \in \Phi$. Define the affine reflection $s_{\alpha,k}$ as follows:

$$
\begin{array}{ccccc}
s_{\alpha,k}  & : & V & \longrightarrow & V \\
                 &   & x & \longmapsto     & x-(2\frac{( \alpha, x )}{( \alpha, \alpha )}-k)\alpha.
\end{array}
$$ 

We consider the subgroup $W_a $ of Aff($V$) generated by all affine reflections $s_{\alpha,k}$ with $\alpha \in \Phi$ and $k \in \mathbb{Z}$, that is 
$$
W_a = \langle s_{\alpha,k}~|~\alpha \in \Phi, ~k \in \mathbb{Z \rangle}.
$$
 The group $W_a$ is called the \emph{affine Weyl group} associated to $\Phi$.  It is also well known \cite[Ch.III, Section 11]{Kane2001} that $W_a = \mathbb{Z}\Phi \rtimes W$. Therefore, any element $w \in W_a$ decomposes as $w=\tau_x\overline{w}$ where $x \in \mathbb{Z}\Phi$ and $\overline{w} \in W$. The element $\overline{w}$  is called the \emph{finite part} of $w$.
 
The classification of irreducible crystallographic root systems states  that there are at most two possible root lengths in $\Phi$. We call short root the shorter ones.  

 Let $\alpha \in \Phi$ such that  $\alpha = a_1\alpha_1 + \cdots + a_n\alpha_n$ with $a_i \in \mathbb{Z}$. The height of $\alpha$ (with respect to $\Delta$) is defined by the number $h(\alpha) = a_1 + \cdots+ a_n$.  Height provides a preorder on $\Phi^+$ defined by $\alpha \leq \beta$ if and only if $h(\alpha) \leq h(\beta)$. 
 
 We denote by $\alpha_0$ the \emph{highest short root} of $\Phi$.
 Then $(W_a, S_a)$ is a Coxeter system (see for example \cite[Ch.4]{Hu90} or \cite{JYS1}).
 
 \medskip
 
 The \emph{inversion set} of $w \in W$ is by definition the set $$N(w) := \{ \xi \in \Phi^+~|~w^{-1}(\xi) \in \Phi^- \}.$$
It is a well known fact that if $w=uv$ is a reduced expression of $w$ (that is $\ell(w) = \ell(u) + \ell(v)$) then the inversion set of $w$ decomposes as \cite[Proposition 2.1]{SRLE}
\begin{equation}\label{N}
 N(w) = N(u) \sqcup u(N(v)).
\end{equation}

Let $\alpha \in \Phi$ and $\alpha^{\vee}:= \frac{2\alpha}{( \alpha, \alpha )}$. For any $k \in \mathbb{Z}$ and any $m \in \mathbb{R}$, we set the hyperplanes 
\begin{align*}
H_{\alpha,k} &= \{x \in V~|~s_{\alpha,k}(x)=x \} \\
&= \{ x \in V~|~ ( x, \alpha^{\vee} ) = k\},
\end{align*}
the half spaces 
$$
H_{\alpha,k}^{^{+}} = \{ x \in V|~ k < ( x,\alpha^{\vee} ) \} 
$$
and
$$
H_{\alpha,k}^{^{-}} = \{ x \in V|~ ( x,\alpha^{\vee}) < k \},
$$

\noindent and the strip
\begin{align*}
H_{\alpha,k}^1 & = \{x \in V~|~k < ( x ,\alpha^{\vee} ) < k+1 \}  \\ &= H_{\alpha,k}^{^{+}} \cap H_{\alpha,k+1}^{^{-}}.
\end{align*}

The connected components of 
$$
 V ~\backslash \bigcup\limits_{\tiny{\begin{subarray}{c}
 ~ ~\alpha \in \Phi^{+} \\ 
  k \in \mathbb{Z}
\end{subarray}}}
H_{\alpha,k} 
$$
are called \emph{alcoves}. We denote $A_e$ the alcove defined as $A_e = \bigcap_{\alpha \in \Phi^+} H_{\alpha,0}^1$. It is well known that $W_a$ acts regularly on the set of alcoves \cite[Ch. 4]{Hu90}. It follows that there is a bijective correspondence between the elements of $W_a$ and all the alcoves. This bijection is defined by $w \mapsto A_w$ where $A_w := wA_e$. We call $A_w$ the corresponding alcove associated to $w \in W_a$. Any alcove of $V$ can be written as an intersection of width-one strips, that is there exists a $\Phi^+$-tuple of integers $(k(w,\alpha))_{\alpha \in \Phi^+}$ such that 
$$
A_w = \bigcap\limits_{\alpha \in \Phi^+}H_{\alpha, k(w,\alpha)}^1.
$$

 In \cite{JYS1} Jian-Yi Shi characterized the elements of any affine Weyl  group $w \in W_a$ by the $ \Phi^+$-tuple of integers $(k(w,\alpha))_{\alpha \in \Phi^+}$ subject to certain conditions.  This characterization is given by the following theorem.
 
\begin{theorem}[\cite{JYS1}, Theorem 5.2]\label{thJYS1} 
Let $A = \bigcap\limits_{\alpha \in \Phi^+} H^1_{\alpha,k_{\alpha}}$ with $k_{\alpha} \in \mathbb{Z}$. Then $A$ is an alcove, if and only if, for all $\alpha$, $\beta \in \Phi^+$ satisfying  $\alpha + \beta \in \Phi^+$, we have the following inequality
\begin{equation}\label{Shi ineq}
||\alpha||^2k_{\alpha} + ||\beta||^2k_{\beta} +1 \leq ||\alpha + \beta||^{2}(k_{\alpha+\beta} +1) \leq ||\alpha||^2k_{\alpha} + ||\beta||^2k_{\beta} + ||\alpha||^2+ ||\beta||^2 + ||\alpha+\beta||^2 -1.
\end{equation}
\end{theorem}

\medskip
 
 Let $P_{\mathcal{H}}$ be the polytope: 
 $$
 P_{\mathcal{H}}:= \bigcap\limits_{\alpha \in \Delta}H_{\alpha,0}^1,
 $$ 
 and let $A_w \subset P_{\mathcal{H}}$. It is clear that $k(w,\alpha) = 0$ for all $\alpha \in \Delta$, and reciprocally, if $w' \in W_a$ is such that $k(w', \alpha) = 0$ for all $\alpha \in \Delta$ then $A_{w'} \subset P_{\mathcal{H}}$. The elements of this polytope seen as $\Phi^+$-tuple of integers are called \emph{admitted}, and more precisely a vector $\lambda \in \bigoplus\limits_{\alpha \in \Phi^+}\mathbb{Z}\alpha$ is admitted if and only if there exists $w \in W_a$ such that $k(w,\alpha) = \lambda_{\alpha}$ for all $\alpha \in \Phi^+$ and such that $A_w \subset P_{\mathcal{H}}$ (see \cite[Section 4.3]{NC1} for a more detailed explication of admitted vectors).

\medskip

\begin{example}
In type $A_n$, an admitted vector $\lambda = (\lambda_{i,j})_{1 \leq i<j\leq n+1}$ is defined by the following conditions:
\begin{equation}\label{condition adm}
\left\{
\begin{array}{ll}
\lambda_{i,j}+\lambda_{j,k} \leq \lambda_{i,k} \leq \lambda_{i,j}+\lambda_{j,k}+1~~\text{~~for all}~i<j<k, \\
\lambda_{i,i+1}=0~~\text{~~for all}~1\leq i<n.
\end{array}
\right.
\end{equation}
\end{example}

\medskip

\subsection{The Shi variety in type $A$}
 In \cite{NC1} the author defines an affine variety $\widehat{X}_{W_a}$, called \emph{the Shi variety} of $W_a$, whose integral points $\widehat{X}_{W_a}(\mathbb{Z})$ are in bijection with $W_a$ \cite[Theorem 4.3]{NC1}.  
 
 The construction of the variety in full generality is not needed in this article,  except in type $A$ where we use in Lemma \ref{stab diamond} the nature of the equations that define $\widehat{X}_{W(\widetilde{A}_n)}$. 
 We briefly recall the construction in type $A_n$ and we refer the reader to  \cite[Section 4]{NC1} for the general construction. 
 
 First,  we can realize a root system of type $ \Phi = A_n$ as follows: Set  $\{e_1,\cdots,e_{n+1}\}$ the canonical basis of $\mathbb{R}^{n+1}$.  Then $\Phi :=\{\pm( e_i-e_j)~|~ 1 \leq i < j \leq n+1 \}$ with $\Phi^+ = \{ e_i-e_j~|~ 1 \leq i < j \leq n+1 \}$ and with simple system $\Delta = \{ e_i-e_{i+1}~|~ 1 \leq i  \leq n \}$.  Let us write for short $k(w,e_i - e_j) = k_{i,j}(w)$.

 Second, from Theorem \ref{thJYS1} we know that each element $w \in W(\widetilde{A}_n)$ is characterized by a $\Phi^+$-tuple $(k_{i,j}(w))_{e_i-e_j\in \Phi^+}$ satisfying a system of inequalities given by (\ref{Shi ineq}).  These inequalities become in type $A_n$ as follows:
 \begin{equation}
k_{i,j}(w)+k_{j,\ell}(w) \leq k_{i,\ell} (w)\leq k_{i,j}(w)+k_{j,\ell}(w)+1~~\text{~~for all~~}~ 1\leq i<j<\ell \leq n+1.
\end{equation}

 In particular we see that:
   $$
   k_{i,\ell}(w) =  k_{i,j}(w)+k_{j,\ell}(w) \text{~}\text{~} \text{~~or~~}  \text{~}\text{~}    k_{i,\ell}(w) =  k_{i,j}(w)+k_{j,\ell}(w) +1.
   $$ 

Iterating this process for any root $\alpha =e_i-e_j \in \Phi^+$, we can express $k_{\alpha}(w)$ in terms of the $k_{\delta}(w)$'s with $\delta \in \Delta$ as follows:
$$
k_{i,j}(w) = k_{i,i+1}(w) + k_{i+1,i+2}(w) + \dots k_{j-1,j}(w) + \lambda_{i,j}(w)
$$
 where $\lambda_{i,j}(w)$ is a positive integer that belongs to $\llbracket 0,  h(e_i-e_j) - 1 \rrbracket$, that is to $\llbracket 0,  j-i- 1 \rrbracket$.

 Setting $X_{i,j}$ to be the formal variable corresponding to the positive root $e_i-e_j$,  the Shi variety $\widehat{X}_{W(\widetilde{A}_n)}$ is defined as the set of solutions in $\mathbb{R}^{n(n+1)/2}$ of the equations:
 \begin{align}\label{equations var A}
  X_{i,i+1} + X_{i+1,i+2} + \dots X_{j-1,j} + \lambda_{i,j} -  X_{i,j} = 0
 \end{align}

 with $1 \leq i <j \leq n+1$ and some particular $\lambda_{i,j} \in \llbracket 0,  j-i- 1 \rrbracket$.  It turns out that the $\lambda_{i,j}$'s appearing in the equations (\ref{equations var A}) are precisely the coefficients of admitted vectors \cite[Section 4.3]{NC1}.
 
\subsection{The irreducible components of the Shi variety}
 It turns out that the irreducible components of this variety are in bijection with the alcoves of the polytope $P_{\mathcal{H}}$ \cite[Proposition 4.1]{NC1}. Writing $H^0(\widehat{X}_{W_a})$ to be the set of irreducible components of  $\widehat{X}_{W_a}$, we have
  $$
  H^0(\widehat{X}_{W_a})= \{\widehat{X}_{W_a}[\lambda]~|~\lambda ~\text{admitted} \}.
  $$
   From the above parameterization and the fact that the components have no intersection we have the following decomposition of $\widehat{X}_{W_a}$ \cite[Theorem 4.3]{NC1}:
$$
 \widehat{X}_{W_a} = \bigsqcup\limits_{\lambda~\text{admitted}} X_{W_a}[\lambda].
 $$
The admitted vectors $\lambda$ are built, in type $A_n$ for example, via the coefficients $\lambda_{i,j}$ appearing in the equations (\ref{equations var A}).
\subsection{The $\Phi^+$-representation}

Let $s_{\alpha, p} \in W_a$.   In \cite{NC1} we defined the affine map $F(s_{\alpha,p})$ as 
 $$
 F(s_{\alpha,p})(x) := L_{\alpha}(x) + v_{p,\alpha}
 $$
  for all $x \in \bigoplus\limits_{\beta \in \Phi^+}\mathbb{R}\beta$,  with $L_{\alpha} \in GL_n(\mathbb{R})$  defined via the matrix $(\ell_{i,j}(\alpha))_{i,j \in \llbracket 1,m \rrbracket}$ where
\begin{equation}\label{matrix F}
 \ell_{\alpha_j, \alpha_i}(\alpha) :=\ell_{j,i}(\alpha) = 
 \left\{
 \begin{array}{rl}
     1  &  \text{if} ~~ s_{\alpha}(\alpha_i) = \alpha_j															\\
    0  &  \text{if} ~~ s_{\alpha}(\alpha_i) \neq \pm \alpha_j												\\
   -1 &  \text{if} ~~ s_{\alpha}(\alpha_i) = -\alpha_j  ,  													\\
 \end{array}
\right.
\end{equation}
 and with $v_{p,\alpha} \in \bigoplus\limits_{\beta \in \Phi^+}\mathbb{R}\beta$ the vector defined by $v_{p,\alpha}=(v_{p,\alpha}(\gamma))_{\gamma \in \Phi^+}$ where 
 \begin{equation}\label{coeff affine F}
v_{p,\alpha}(\gamma) := 
 \left\{
 \begin{array}{rl}
    -p(\alpha , s_{\alpha}(\gamma)^{\vee}) &  \text{if} ~~ s_{\alpha}(\gamma) \in \Phi^+  					\\
        -1-p(  \alpha, s_{\alpha}(\gamma)^{\vee} ) &  \text{if} ~~ s_{\alpha}(\gamma) \in \Phi^-.    			\\
 \end{array}
\right.
 \end{equation}

 \medskip

For $w \in W_a$ we denote $L_w $ to be the left multiplication by $w$. In \cite{NC1} we showed that $F$ extends naturally to $W_a$. We also showed that $F$ induces a geometrical action on the irreducible components.

\begin{theorem}[\cite{NC1}, Theorem 3.1] \label{th Phi rep}
 There exists an injective morphism ${F : W_a \rightarrow Isom(\mathbb{R}^m)}$ such that for any $w \in W_a $ the following diagram commutes. This morphism is called the $\Phi^+$-representation of $W_a$, and the  corresponding action is called the $\Phi^+$-action of $W_a$.
$$
  \xymatrix{
 W_a \ar[r]^{L_w} \ar@{^{(}->}[d]_{\iota} &  W_a \ar@{^{(}->}[d]_{\iota} \\
    \mathbb{R}^m \ar[r]_{F(w)}                       & \mathbb{R}^m.
  }
$$
\end{theorem}

\begin{proposition}[\cite{NC1}, Proposition 4.3 ]\label{action isom}
Let $F : W_a \hookrightarrow$ Isom$(\mathbb{R}^n)$ be the $\Phi^+$-representation of $W_a$. Then we have:
\begin{itemize}
\item[1)] $W_a$ acts naturally on the irreducible components of $\widehat{X}_{W_a}$ via the action defined as $w\diamond X_{W_a}[\lambda] := F(w)(X_{W_a}[\lambda])$. Furthermore if we assume that $w\in W_a$ decomposes as ${w=\tau_x\overline{w}}$,  then ${w \diamond X_{W_a}[\lambda]=\overline{w}\diamond X_{W_a}[\lambda]}$. Finally this action is transitive.

\item[2)] The previous action induces an action on the admitted vectors by $w\diamond \lambda := \gamma$ such that $w\diamond X_{W_a}[\lambda] = X_{W_a}[\gamma]$. In other words we have $w\diamond X[\lambda]=X[w \diamond \lambda]$.
\end{itemize}
\end{proposition}

\subsection{The main theorem of the article}
In this article we thoroughly investigate $H^0(\widehat{X}_{W(\widetilde{A}_n)})$ from a combinatorial point of view. The main result of this article is Theorem \ref{Bijection entre pyras et permutations}. Let us first recall some basics about the symmetric group:

 We call \emph{circular permutations} the $(n+1)$-cycles  of $ S_{n+1}=W(A_{n})$. The action by conjugation of $W(A_n)$ on itself is defined for all $\sigma, \gamma \in W(A_n)$ by ${\sigma.\gamma :=\sigma \gamma \sigma^{-1}}$. This action appears in a lot of areas and has been studied many times. In particular,  understanding the orbits of the action, which are the conjugacy classes, yielded a lot of research work. For example, M.  Geck and G.  Pfeiffer used in  \cite{GP} the technology of cuspidal class in order to express any conjugacy class in terms of these cuspidal classes.
 
  We relate in this article the conjugacy class of $(1~2~\cdots~n+1)$ with the irreducible components of the Shi variety corresponding to $W(\widetilde{A}_n)$. Finally, the action by conjugation plays a crucial role in the following theorem:

\begin{theorem}\label{Bijection entre pyras et permutations}
There is a natural bijection between $H^0(\widehat{X}_{W(\widetilde{A}_n)})$ and the circular permutations of $W(A_n)$. In particular $|H^0(\widehat{X}_{W(\widetilde{A}_n)})|=n!$.
\end{theorem}

In Theorem \ref{Bijection entre pyras et permutations} the word \say{natural} is used because the bijection involved respects the action of $W(A_n)$ on two different sets, namely the action we have defined on $H^0(\widehat{X}_{W(\widetilde{A}_n)})$ and the conjugation action on the circular permutations in $W(A_n)$.
Because of Proposition \ref{action isom} we know that the components are invariant under translations, that is only the finite part matters. Consequently, we will only look at the action of the finite part $W(A_n)$. Thus, the goal is to understand how the components behave when we apply $F(w)$ on them for any $w \in W(A_n)$. 

Giving explicit formulas for the action of $W(A_n)$ on $H^0(\widehat{X}_{W(\widetilde{A}_n)})$ is of particular interest as well. These formulas are established in Theorem \ref{changement coordonnees}.

\section{Bijection between $H^0(\widehat{X}_{W(\widetilde{A}_n)})$ and the set of circular permutations of $W(A_n)$ }

\subsection{Notations and setup}\label{notations}
Let $E$ be a $\mathbb{R}$-vector space and let $G$ be a group that acts linearly on $E$. This action induces a linear action on the $r$th exterior power $\bigwedge\nolimits^r(E)$ for all $r \in \mathbb{N}$. We call this action \emph{the diagonal action} on $\bigwedge\nolimits^r(E)$.
 
 For instance, if $G=W(A_n)$ is the symmetric group and $E$ is the vector space $E:=\text{span}(e_1,\cdots,e_{n+1})$ then $G$ acts on $E$ by permutation of the coordinates, that is $\sigma.e_i = e_{\sigma(i)}$. Then the diagonal action of $G$ on $\bigwedge\nolimits^2(E)$ is given by 
 $$
 \sigma \cdot (e_{k} \wedge e_{\ell})  =e_{\sigma(k)} \wedge e_{\sigma(\ell)}.
 $$

From now on $\{e_1,\cdots,e_{n+1}\}$ represents the canonical basis of $K:=\mathbb{R}^{n+1}$.  Recall that a way  to represent the root system $A_n$ in $K$ is by $\Phi :=\{\pm( e_i-e_j)~|~ 1 \leq i < j \leq n+1 \}$ with $\Phi^+ = \{ e_i-e_j~|~ 1 \leq i < j \leq n+1 \}$ and with simple system $\Delta = \{ e_i-e_{i+1}~|~ 1 \leq i  \leq n \}$.  We recall that $m=|\Phi^+|$ and we set $Y$ to be the $\mathbb{R}$-vector space with basis $e_{i,j}$ for $1 \leq i<j \leq n+1$. The elements $e_{i,j}$ are in obvious bijection with the positive roots $e_i-e_j$. We denote by  $s_{k,\ell}$ the reflection $s_{e_k-e_{\ell},0}$, namely $s_{k,\ell}$ is the transposition of $S_{n+1}$ that swaps $k$ and $\ell$. Finally, we will denote $L_{k,\ell} := L_{e_k-e_{\ell}}$.

\subsection{Affine diagonal action}

\begin{definition}\label{def action}
We define the \emph{affine diagonal action} of $W(A_n)$ on $\bigwedge\nolimits^2(K)$ as follows. Write $W(A_n) =\langle s_{1},\cdots,s_{n} \rangle$ where $s_{i}$ is the adjacent transposition $(i, i+1)$. We define the operation $\odot$ on the generators $s_i$ and on the basis of $\bigwedge\nolimits^2(K)$, where for all $k < \ell$ we have
 \begin{align}\label{def_action}
  s_{i}\odot e_k \wedge e_{\ell} & := (i,i+1)\cdot (e_k \wedge e_{\ell}) - e_i \wedge e_{i+1}  = e_{s_i(k)} \wedge e_{s_i(\ell)} - e_i \wedge e_{i+1},
 \end{align}
and we extend it as follows 
\begin{equation} \label{action_gen}
s_i \odot (\sum\limits_{r<s}x_{r,s}e_r \wedge e_s) = (\sum\limits_{r<s}x_{r,s}e_{s_i(r)} \wedge e_{s_i(s)}) -e_i \wedge e_{i+1}.
\end{equation}

 \end{definition}
 
 \begin{proposition}\label{action rond}
 The operation $\odot : W(A_n) \times \bigwedge\nolimits^2(K) \rightarrow \bigwedge\nolimits^2(K)$ is an action which is not linear.
 \end{proposition}

\begin{proof}
In order to prove this statement we just need to check that the relations of $W(A_n)$ are preserved under the operation $\odot$. It turns out that this is the case because
\begin{align*}
& \text{~} \text{~} \text{~} \text{~} s_i \odot (s_{i+1} \odot (s_i \odot (e_k \wedge e_{\ell})))\\ & = s_i \odot (s_{i+1}\odot (e_{s_i(k)} \wedge e_{s_i(\ell)} - e_i \wedge e_{i+1})) \\
                                                  & = s_i\odot (e_{s_{i+1}s_i(k)} \wedge e_{s_{i+1}s_i(\ell)} - e_{s_{i+1}(i)} \wedge e_{s_{i+1}(i+1)} - e_{i+1} \wedge e_{i+2}) \\
                                                  & = s_i\odot (e_{s_{i+1}s_i(k)} \wedge e_{s_{i+1}s_i(\ell)} - e_{i} \wedge e_{i+2} - e_{i+1} \wedge e_{i+2}) \\
                                                  & = e_{s_is_{i+1}s_i(k)} \wedge e_{s_is_{i+1}s_i(\ell)} - e_{s_i(i)} \wedge e_{s_i(i+2)} - e_{s_i(i+1)} \wedge e_{s_i(i+2)}-e_i \wedge e_{i+1}\\
                                                  & =  e_{s_is_{i+1}s_i(k)} \wedge e_{s_is_{i+1}s_i(\ell)} - e_{i+1} \wedge e_{i+2} - e_{i} \wedge e_{i+2}-e_i \wedge e_{i+1},
\end{align*}
and
\begin{align*}
& \text{~} \text{~} \text{~} \text{~}  s_{i+1} \odot (s_{i} \odot (s_{i+1} \odot (e_k \wedge e_{\ell}))) \\& = s_{i+1} \odot (s_{i}\odot (e_{s_{i+1}(k)} \wedge e_{s_{i+1}(\ell)} - e_{i+1} \wedge e_{i+2})) \\
                                                  & = s_{i+1}\odot (e_{s_{i}s_{i+1}(k)} \wedge e_{s_{i}s_{i+1}(\ell)} - e_{s_{i}(i+1)} \wedge e_{s_{i}(i+2)} - e_{i} \wedge e_{i+1}) \\
                                                  & = s_{i+1}\odot(e_{s_{i}s_{i+1}(k)} \wedge e_{s_{i}s_{i+1}(\ell)} - e_{i} \wedge e_{i+2} - e_{i} \wedge e_{i+1}) \\
                                                  & = e_{s_{i+1}s_{i}s_{i+1}(k)} \wedge e_{s_{i+1}s_{i}s_{i+1}(\ell)} - e_{s_{i+1}(i)} \wedge e_{s_{i+1}(i+2)} ~- \\
                                                  & ~~~~e_{s_{i+1}(i)} \wedge e_{s_{i+1}(i+1)}-e_{i+1} \wedge e_{i+2}\\
                                                  & =  e_{s_is_{i+1}s_i(k)} \wedge e_{s_is_{i+1}s_i(\ell)} - e_{i} \wedge e_{i+1} - e_{i} \wedge e_{i+2}-e_{i+1} \wedge e_{i+2}.
\end{align*}

Thus the mesh relation $s_is_{i+1}s_i = s_{i+1}s_is_{i+1}$ is preserved under this action. For $|i-j| \geq 2$ we know that $s_is_j = s_js_i$. Let us show that this relation is also preserved. This is the case because
\begin{align*}
s_i\odot (s_j\odot (e_k \wedge e_{\ell})) & = s_i\odot (e_{s_j(k)} \wedge e_{s_j(\ell)}-e_j\wedge e_{j+1}) \\
										   & = e_{s_is_j(k)} \wedge e_{s_is_j(\ell)}-e_{s_i(j)}\wedge e_{s_i(j+1)}-e_i \wedge e_{i+1} \\
										   & = e_{s_is_j(k)} \wedge e_{s_is_j(\ell)}-e_j\wedge e_{j+1}-e_i \wedge e_{i+1},
\end{align*}
and
\begin{align*}
s_j\odot (s_i\odot (e_k \wedge e_{\ell})) & = s_j\odot (e_{s_i(k)} \wedge e_{s_i(\ell)}-e_i\wedge e_{i+1}) \\
										   & = e_{s_js_i(k)} \wedge e_{s_js_i(\ell)}-e_{s_j(i)}\wedge e_{s_j(i+1)}-e_j \wedge e_{j+1} \\
										   & = e_{s_js_i(k)} \wedge e_{s_js_i(\ell)}-e_{i}\wedge e_{i+1}-e_j \wedge e_{j+1}.
\end{align*}
There is one type of relation left to check: those coming from the involutions. Thus
\begin{align*}
s_i \odot (s_i \odot e_k \wedge e_{\ell}) & = s_i \odot (e_{s_i(k)} \wedge e_{s_i(\ell)} - e_i \wedge e_{i+1} ) \\
													       & = e_{s_is_i(k)} \wedge e_{s_is_i(\ell)} - e_{s_i(i)} \wedge e_{s_i(i+1)} - e_i \wedge e_{i+1} \\
													       & = e_k \wedge e_{\ell} -e_{i+1} \wedge e_i - e_i \wedge e_{i+1} \\
													       & = e_k \wedge e_{\ell} +e_{i} \wedge e_{i+1} - e_i \wedge e_{i+1}= e_k \wedge e_{\ell}.
\end{align*}
\end{proof}

\begin{definition}
We define $\Theta$ from $Y$ to $\bigwedge\nolimits^2(K)$ to be the linear map sending the elements $e_{i,j}$ of the basis of $Y$ to the elements $e_i \wedge e_j$ of the basis of  $\bigwedge\nolimits^2(K)$.
\end{definition}

\begin{proposition}\label{isomorphisme action}
The map $\Theta$ is an isomorphism from the $ \Phi^+$-action of $W(A_n)$ on $Y$ to the affine diagonal action of $W(A_n)$ on $\bigwedge^2(K)$.
\end{proposition}

\begin{proof}
We need to show that the following diagram commutes  for all  $w \in W(A_n)$
$$
 \xymatrix{
    Y \ar[r]^{F(w)} \ar[d]^{\wr}_{\Theta}  &  Y \ar[d]_{\wr}^{\Theta}\\
  \bigwedge\nolimits^2(K) \ar[r]_{\varphi_{w}} & \bigwedge\nolimits^2(K)
  }
  $$
where $\varphi_{w}(e_i\wedge e_j):=w \odot e_i \wedge e_j$.  However, it is enough to only show the commutativity of this diagram for the generators $s_i \in S$. First of all, $\Theta$ is an isomorphism because the set $\{e_{i,j}~|~1\leq i<j \leq n+1 \}$ is a basis of $Y$, and  ${\{e_i \wedge e_j ~|~1 \leq i<j \leq n+1\}}$ is a basis of $\bigwedge\nolimits^2(K)$. Let $1 \leq k < \ell \leq n+1$. With respect to the commutativity we have  
  \begin{align*}
   \varphi_{s_i} \circ \Theta(e_{k,\ell}) & =  \varphi_{s_i}(e_k \wedge e_{\ell}) = e_{s_i(k)}\wedge e_{s_i(\ell)} - e_i \wedge e_{i+1}.
  \end{align*}

 Moreover, it follows from (\ref{matrix F}) and (\ref{coeff affine F}) that $F(s_i)(e_{k,\ell}) = e_{s_i(k),s_i(\ell)} - e_{i,i+1}$. Therefore
\begin{align*}
  \Theta \circ F(s_i)(e_{k,\ell}) & = \Theta(e_{s_i(k),s_i(\ell)} - e_{i,i+1}) = \Theta(e_{s_i(k),s_i(\ell)}) - \Theta(e_{i,i+1})  \\ 
  & = e_{s_i(k)} \wedge e_{s_i(\ell)} - e_i \wedge e_{i+1}.
\end{align*}

\end{proof}

\begin{corollary}\label{action sur reflexion}
Let $w \in W(A_n)$  and $1 \leq k < \ell \leq n+1$. Then we have $$w \odot e_k \wedge e_{\ell} = e_{w(k)} \wedge e_{w(\ell)} - \sum\limits_{e_r-e_s \in N(w)} e_r \wedge e_s.$$
\end{corollary}

\begin{proof}
We proceed by induction on the length. The base case follows from (\ref{def_action}). Let us write ${w:=s_{i_1}s_{i_2}\cdots s_{i_p}:=s_{i_1}w'}$ be a reduced expression of $w$. Let us also assume that $w' \odot e_k \wedge e_{\ell} = e_{w'(k)} \wedge e_{w'(\ell)} - \sum\limits_{e_r-e_s \in N(w')} e_r \wedge e_s$. Because of (\ref{N}) we know that 
\begin{align}\label{N'}
N(w) &= N(s_{i_1}) \sqcup s_{i_1}N(w') = \{e_{i_1}-e_{i_1+1}\} \sqcup s_{i_1}N(w').
\end{align}

Therefore, it follows that
\begin{align*}
s_{i_1}w' \odot e_k \wedge e_{\ell} & ~= s_{i_1}\odot (w' \odot e_k \wedge e_{\ell})  =  s_{i_1} \odot (e_{w'(k)} \wedge e_{w'(\ell)} - \sum\limits_{e_r-e_s \in N(w')} e_r \wedge e_s ) \\
												  & \stackrel{(\ref{action_gen})}{=} e_{s_{i_1}w'(k)} \wedge e_{s_{i_1}w'(\ell)} - \sum\limits_{e_r-e_s \in N(w')} e_{s_{i_1}(r)} \wedge e_{s_{i_1}(s)} -e_{i_1} \wedge e_{i_1+1} \\
												  & \stackrel{(\ref{N'})}{=} e_{w(k)} \wedge e_{w(\ell)} -\sum\limits_{e_r-e_s \in N(w)} e_r \wedge e_s.
\end{align*}
\end{proof}

\subsection{Proof of the main theorem}

\begin{lemma}\label{stab diamond}
$(1~2~\cdots~n+1)\diamond X_{W(\widetilde{A}_n)}[0] = X_{W(\widetilde{A}_n)}[0]$.
\end{lemma}

\begin{proof}
Let us write $E_{i,j}$ for the equation $X_{i,j}- \sum\limits_{r=i}^{j-1} X_{r,r+1}=0$ for $1 \leq i < j \leq n+1$, and $E_{i,j}'$ for the equation $X_{1,j}-X_{1,i+1} - \sum\limits_{r=i+1}^{n} X_{r,r+1}=0$ for $1 <i < j \leq n+1$.  We know that $X_{W(\widetilde{A}_n)}[0]$ is the component of $\widehat{X}_{W(\widetilde{A}_n)}$ cut out by the equations $E_{i,j}$. 
We analyze the action of ${(1~2~\cdots~n+1)}$ on $X_{W(\widetilde{A}_n)}[0]$ by studying the action of $(1~2~\cdots~n+1)$ on the equations $E_{i,j}$, with the relations $X_{i,j} = -X_{j,i}$ for $i <j$. The equations that cut out ${(1~2~\cdots~n+1)\diamond X_{W(\widetilde{A}_n)}[0]}$ are exactly the $ (1~2~\cdots~n+1) \cdot E_{i,j}$, and a short computation shows that

\begin{equation}\label{action equations}
(1~2~\cdots~n+1) \cdot E_{i,j} = 
\left\{
\begin{array}{ll}
E_{i+1,j+1} \text{~}\text{~}\text{if} ~j \neq n+1 \\
E_{1,n+1} \text{~}\text{~}\text{~}\text{~}\text{if} ~i=1,~j = n+1 \\
E_{i,j}'\text{~}\text{~}\text{~}\text{~}\text{~}\text{~}\text{~}\text{~}\text{if}~1 <i,~j=n+1.
\end{array}
\right.
\end{equation}

Thus, some of the equations that determine the component $X_{W(\widetilde{A}_n)}[0]$ are just swapped, namely the $E_{i,j}$ with $j  < n+1$, and $E_{1,n+1}$. The other ones are sent to some equations that are not in the set of equations cutting out $X_{W(\widetilde{A}_n)}[0]$. 

However $(1~2~\cdots~n+1)$ acts as an homeomorphism on $Y$. Thus the set ${(1~2~\cdots~n+1)\diamond X_{W(\widetilde{A}_n)}[0]}$ must be an irreducible component of $X_{W(\widetilde{A}_n)}$. Thanks to (\ref{action equations}) we see that $0 \in  (1~2~\cdots~n+1)\diamond X_{W(\widetilde{A}_n)}[0]$. Since there is an only one component in $\widehat{X}_{W(\widetilde{A}_n)}$ that contains 0, namely $X_{W(\widetilde{A}_n)}[0]$, it follows that $(1~2~\cdots~n+1)\diamond X_{W(\widetilde{A}_n)}[0]= X_{W(\widetilde{A}_n)}[0].$
\end{proof}

\medskip

We are now ready to prove our main result.

\medskip

\begin{proof}[Proof of Theorem \ref{Bijection entre pyras et permutations}]
 It is well known that in the symmetric group $W(A_n)$ the stabilizer of any circular permutation for the conjugation action is the subgroup generated by this permutation. More precisely, let $\sigma$ be a circular permutation of $W(A_n)$. Then $Stab_{W(A_n)}(\sigma) = \langle \sigma \rangle$. It is also clear that this action acts transitively on the circular permutations. Let us consider now the action of $W(\widetilde{A}_n)$ on $H^0(\widehat{X}_{W(\widetilde{A}_n)})$. Since ${W(\widetilde{A}_n) = \mathbb{Z}\Phi \rtimes W(A_n)}$ we also have $W(A_n)$ that acts on $H^0(\widehat{X}_{W(\widetilde{A}_n)})$.  Moreover, because of  Lemma \ref{stab diamond} we know that the circular permutation $(1~2~\cdots~n+1)$ of $W(A_n)$ is such that  $(1~2~\cdots~n+1)\diamond X_{W(\widetilde{A}_n)}[0] = X_{W(\widetilde{A}_n)}[0]$. Then, the subgroup generated by $(1~2~\cdots~n+1)$ is included in $Stab_{\diamond}(X_{W(\widetilde{A}_n)}[0])$. Besides, by Proposition \ref{action isom} we know that the action $\diamond$ is transitive. Theorem 4.3 of \cite{NC1} tells us that $|H^0(\widehat{X}_{W(\widetilde{A}_n)})|=n!$. Thus it follows
$$
\left | \faktor{W(A_n)}{ Stab_{\diamond}(X_{W(\widetilde{A}_n)}[0])} \right |= |H^0(\widehat{X}_{W(\widetilde{A}_n)})|=n!.
$$
Then, one has $|Stab_{\diamond}(X_{W(\widetilde{A}_n)}[0])| =  n+1$ and it follows 
$$
Stab_{\diamond}(X_{W(\widetilde{A}_n)}[0]) = \langle (1~2~\cdots~n+1) \rangle.
$$
Therefore, the map given by $ (1~2~\cdots~n+1) \mapsto X_{W(\widetilde{A}_n)}[0]$ induces a bijective map defined as follows
$$
\begin{array}{cccc}
\{\text{Circular permutations of $W(A_n)$}\}  & \longrightarrow & H^0(\widehat{X}_{W(\widetilde{A}_n)}) \\
       \sigma(1~2~\cdots~n+1)\sigma^{-1}  & \longmapsto     &  \sigma\diamond X_{W(\widetilde{A}_n)}[0].
\end{array}
$$
\end{proof}

In Figures \ref{position pyra}, \ref{image entre posets 1}, \ref{poset pyra circ A4} and \ref{poset permutation circ A4} we describe the two posets respectively associated to $W(\widetilde{A}_3)$ and $W(\widetilde{A}_4)$.

When we express a vector $v \in \bigoplus_{\alpha \in A_{n}^+}\mathbb{R}\alpha$ as a row vector, or column vector, the way to proceed is by adding, from left to right and from bottom to top, the diagonals of the triangle. Therefore, the coefficient $v_{ij}$ in the below triangle is the coordinate of $v$ in position $\alpha = e_i-e_j$. For example in Figure \ref{position pyra} this expression is as follows: $v = (v_{12}, v_{13}, v_{14}, v_{23}, v_{24}, v_{34})$.

\begin{figure}[h!]
\centering
\includegraphics[scale=0.8]{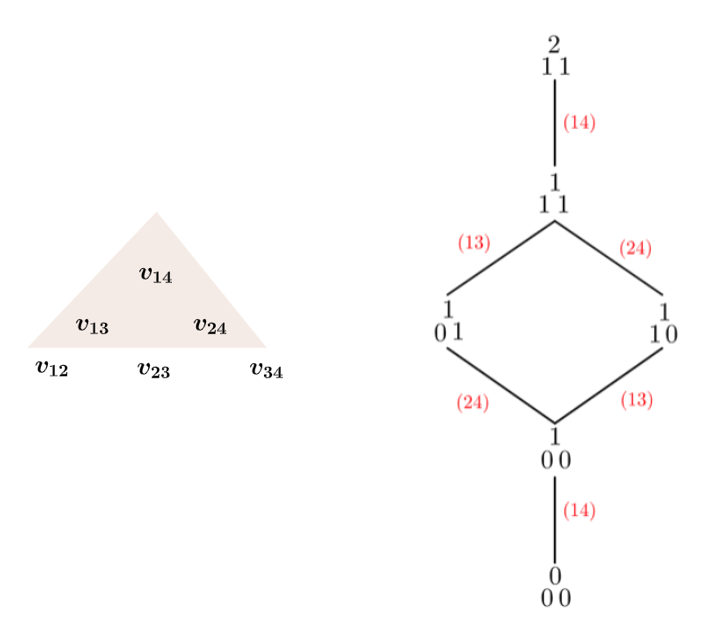} 
\caption{Poset associated to $\widehat{X}_{W(\widetilde{A}_3)}$. The coordinates on the simple roots are erased since they are all equal to 0. The red labels represent the natural order on $\mathbb{Z}^6$, that is the red label on an edge indicates the positive root whose coordinate is increased when going up the edge.} 
\label{position pyra}
\end{figure}
\bigskip
\vspace{2cm}
\bigskip
\begin{figure}[h!]
\centering
\includegraphics[scale=0.8]{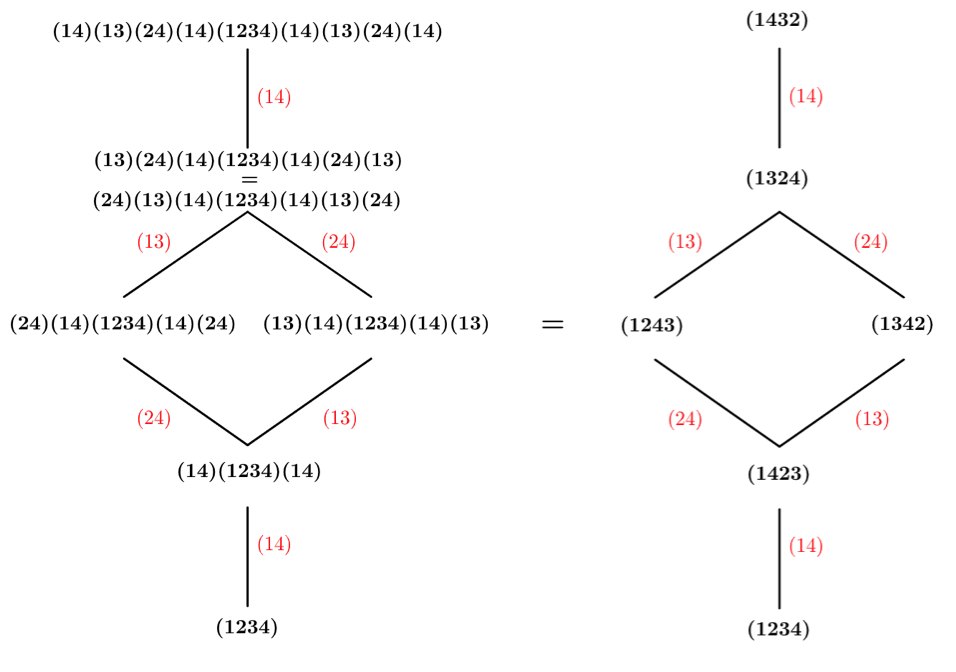} 
\caption{Poset of admitted vectors of $W(\widetilde{A}_3)$ on the left hand side and circular permutations of $W(A_3)$ on the right hand side. The red labels indicate the cover relation which is the conjugation action.} 
\label{image entre posets 1}
\end{figure}

\newpage

\begin{figure}[h!]
\centering
\includegraphics[scale=0.8]{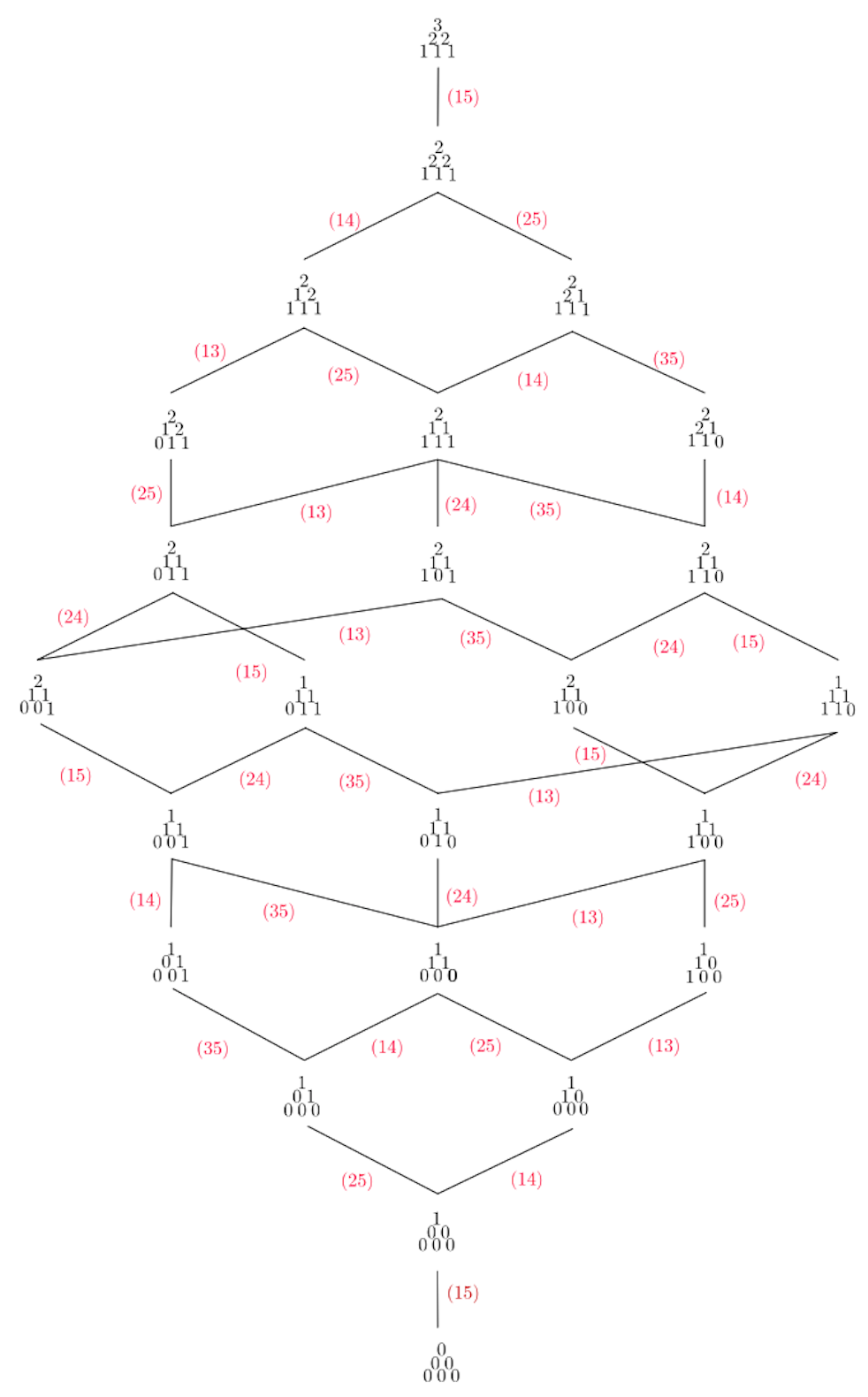} 
\caption{Poset of admitted vectors of $W(\widetilde{A}_4)$. The four coordinates on the simple roots are erased since they are all equal to 0. The red labels represent the natural order on $\mathbb{Z}^{10}$, that is the red label on an edge indicates the positive root whose coordinate is increased when going up the edge.}
\label{poset pyra circ A4}
\end{figure}

\newpage

\begin{figure}[h!]
\centering
\includegraphics[scale=0.88]{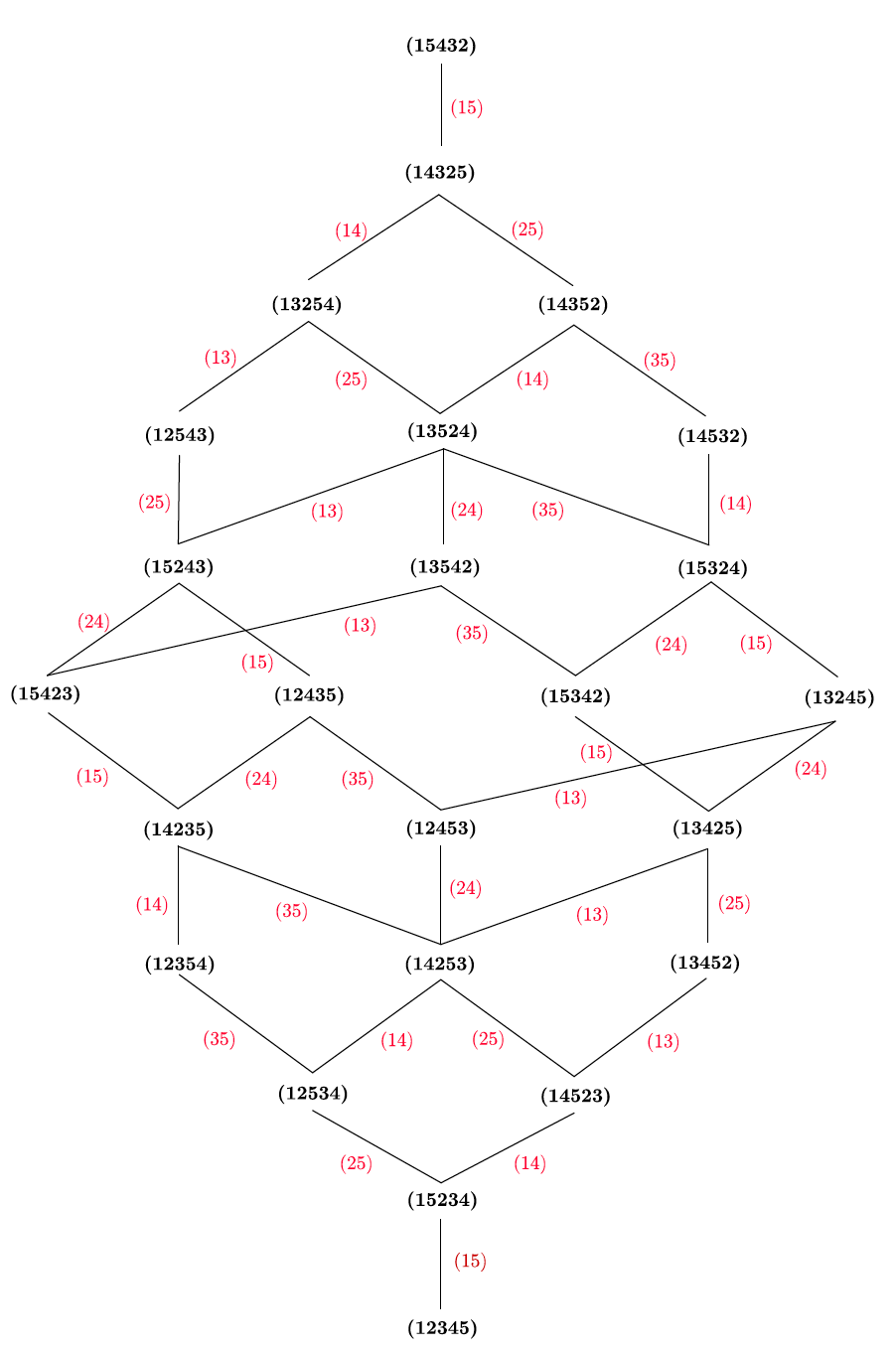} 
\caption{Poset of circular permutations of $W(\widetilde{A}_4)$. The red labels indicate the cover relation: this is the conjugation of the lower circular permutation by the red label.}
\label{poset permutation circ A4}
\end{figure}

\begin{remark}
In \cite{abram2020order} we thoroughly investigated the structure of $H^0(\widehat{X}_{W(\widetilde{A}_n)})$ from another combinatorial point of view, and we gave a deeper explanation of the above bijection. We have shown in particular that this poset map is an isomophism of posets. We also showed that these posets are semidistributive lattices and we provided a way to compute the join of any pair of elements.
\end{remark}

\section{Investigation of the $\diamond$ action on $H^0(\widehat{X}_{W(\widetilde{A}_n)})$}

\subsection{Motivation and example}\label{motivation}

In \cite{NC1} a concrete description of the coefficients $k(w,\alpha)$ for $\alpha \in \Phi^+\setminus \Delta$ can be found in any type in terms of the coefficients $k(w,\delta)$ with $\delta \in \Delta$. We  use here this description in type $A$.

 Let $w \in W(\widetilde{A}_n)$ and let $1 \leq i < j \leq n+1$. We denote  $k(w, \alpha) = k_{i,j}(w)$ where $\alpha = e_i - e_j$ is a positive root of $A_n$. According to Theorem 4.1 and Lemma 4.1 of \cite{NC1}, there exists a positive integer $\lambda_{i,j}$ such that the coefficient $k_{i,j}(w)$ is given by the formula: 
$$
k_{i,j}(w) = \sum\limits_{r=i}^{j-1} k_{r,r+1}(w)+ \lambda_{i,j}.
$$

It is precisely from these formulas that the Shi variety $\widehat{X}_{W(\widetilde{A}_n)}$ is defined in \cite{NC1}. Therefore, if $\lambda=(\lambda_{i,j})$ is an admitted vector and if $x=(x_{i,j}) \in X_{W(\widetilde{A}_n)}[\lambda]$, because of the construction of the irreducible components we must have for all $1 \leq i < j \leq n+1$ that: 
\begin{equation}\label{eq var x}
x_{i,j} = \sum\limits_{r=i}^{j-1}x_{r,r+1}+\lambda_{i,j}.
\end{equation}

Since $W(\widetilde{A}_n) \simeq \mathbb{Z}A_n \rtimes W(A_n)$, one can express $w$ as $w=\tau_u \overline{w}$ where $u \in \mathbb{Z}A_n$ and $\overline{w} \in W(A_n)$. The goal of this section is to understand, in a practical way, the following action
$$
\begin{array}{ccc}
 W(\widetilde{A}_n) \times H^0(\widehat{X}_{W(\widetilde{A}_n)})& \longrightarrow & H^0(\widehat{X}_{W(\widetilde{A}_n)}) \\
 (w, X_{W(\widetilde{A}_n)}[\lambda]) & \longmapsto & w \diamond X_{W(\widetilde{A}_n)}[\lambda].
\end{array}
$$

 It is of particular interest to provide convenient formulas for the action $\diamond$ in order to better understand the poset structure, and more specially its cover relation.

We know that the irreducible components of $\widehat{X}_{W(\widetilde{A}_n)}$ are parameterized by the admitted vectors of $W(\widetilde{A}_n)$ (see \cite{NC1} Theorem 4.3). Therefore it is equivalent to look at this action on the set of admitted vectors. Let $\lambda$ be an admitted vector. Because of Theorem \ref{action isom} we know that the action by translation has no effect, that is $w \diamond \lambda = \overline{w} \diamond \lambda$. Thus, it is enough to understand this action restricted to $W(A_n)$.

Let us then take $g \in W(A_n)$ and let us denote $\beta = g \diamond \lambda$. Our purpose is to express $\beta$ in terms of $\lambda$. The way to proceed is as follows.
 $W(A_n)$ acts on the integral points of $X_{W(\widetilde{A}_n)}[\lambda]$ via the $\Phi^+$-representation. Hence, in order to determine the admitted vector $\beta$, we just have to find which component contains the element $F(g)(x)$. 
 
Recall that $\{e_{i,j}~|~1\leq i<j \leq n+1 \}$ is a basis of $Y$. Therefore one has
 $$
 F(g)(x) = F(g)(\sum\limits_{i<j}x_{i,j}e_{i,j}) =F(g)(\sum\limits_{i<j}(\sum\limits_{r=i}^{j-1}x_{r,r+1}+\lambda_{i,j})e_{i,j}).
 $$ 
 Finally, the goal is to find the expression of $\beta$ in terms of $\lambda$ from the equation:
 \begin{equation*}
 F(g)(\sum\limits_{i<j}(\sum\limits_{r=i}^{j-1}x_{r,r+1}+\lambda_{i,j})e_{i,j}) = \sum\limits_{i<j}(\sum\limits_{r=i}^{j-1}y_{r,r+1}+\beta_{i,j})e_{i,j}.
 \end{equation*}

\bigskip

\begin{example}
Let us take the group $W(A_3)$. Because of Theorem \ref{Bijection entre pyras et permutations} we know that $\widehat{X}_{W(\widetilde{A}_3)}$ has 6 irreducible components. We think of these components as admitted vectors and we delete the $\Delta$-part since these coordinates are zero. 

Using (\ref{condition adm}), a short computation shows that these vectors (with positions $[\lambda_{13},\lambda_{14}, \lambda_{24}]$) are $$\{ [0,0,0], [0,1,0], [0,1,1], [1,1,1], [1,1,0], [1,2,1] \}.$$ 

Let $g=s_{1,2} \in W(A_3)$ be the reflection corresponding to the positive root $e_1-e_2$, $\lambda$ be an admitted vector and $x \in X_{W(\widetilde{A}_3)}[\lambda]$. Let  $\beta=(\beta_{i,j})$ be the admitted vector such that $F(g)(x) \in X_{W(\widetilde{A}_3)}[\beta]$. Because of (\ref{matrix F}), (\ref{coeff affine F}) and (\ref{eq var x}) it is easy to see that the matrix representation of the affine map $F(g)$. The next computation gives the expression of $F(g)(x)$:

  \begin{footnotesize}
$$
\begin{pmatrix} 
    -1  & 0  & 0 & 0  & 0 & 0  \\
     0  & 0  & 0 & 1  & 0  & 0  \\
     0  & 0  & 0 & 0  & 1  &  0  \\
     0  & 1  & 0 & 0  & 0  &  0  \\
     0  & 0  & 1 & 0  & 0  &  0  \\
     0  & 0  & 0 & 0  & 0  &  1  \\
\end{pmatrix}
\begin{pmatrix}
x_{12} \\
 x_{12}+x_{23}+\lambda_{13} \\
 x_{12}+x_{23}+x_{34}+\lambda_{14} \\
x_{23} \\ 
x_{23}+x_{34}+\lambda_{24}   \\
x_{34}\\
\end{pmatrix}
+
\begin{pmatrix}
-1 \\
 0 \\
0 \\ 
0 \\
0\\
 0 \\ 
\end{pmatrix}
=
\begin{pmatrix}
 -x_{12}-1\\
 x_{23} \\
x_{23}+x_{34}+\lambda_{24} \\ 
  x_{12}+x_{23}+\lambda_{13}\\
 x_{12}+x_{23}+x_{34}+\lambda_{14}\\
 x_{34} \\ 
\end{pmatrix}.
$$
 \end{footnotesize}
 
Let us denote $F(g)(x) = y$. Once again because of (\ref{eq var x}) it follows that

 \begin{footnotesize}
$$
\begin{pmatrix}
 -x_{12}-1\\
 x_{23} \\
x_{23}+x_{34}+\lambda_{24} \\ 
  x_{12}+x_{23}+\lambda_{13}\\
 x_{12}+x_{23}+x_{34}+\lambda_{14}\\
 x_{34} \\ 
\end{pmatrix}
=
\begin{pmatrix}
y_{12} \\
 y_{12} + y_{23}+ \beta_{13} \\
y_{12}+ y_{23}+ y_{34}+ \beta_{14}\\ 
y_{23}\\
y_{23} +y_{34}+ \beta_{24}\\
 y_{34} \\ 
\end{pmatrix}.
$$
 \end{footnotesize}
 
Thus one has
$$
 \left\{
\begin{array}{ll}
y_{12} = -x_{12}-1  \\
y_{23}=x_{12}+x_{23}+\lambda_{13}\\
y_{34}=x_{34}
\end{array}
\right. 
~~~~\text{and}~~~~~
 \left\{
\begin{array}{ll}
x_{23} =  y_{12} + y_{23}+ \beta_{13} \\
x_{23}+x_{34}+\lambda_{24}=y_{12}+ y_{23}+ y_{34}+ \beta_{14}\\
x_{12}+x_{23}+x_{34}+\lambda_{14}=y_{23} +y_{34}+ \beta_{24}.
\end{array}
\right. 
$$
It follows that
$$
 \left\{
\begin{array}{ll}
\beta_{13} = x_{23} -(-x_{12}-1) -(x_{12}+x_{23}+\lambda_{13})= -\lambda_{13}+1\\
\beta_{14} = x_{23}+x_{34}+\lambda_{24}-(-x_{12}-1) -(x_{12}+x_{23}+\lambda_{13})-x_{34}=-\lambda_{13}+\lambda_{24}+1 \\
 \beta_{24}= x_{12}+x_{23}+x_{34}+\lambda_{14} -(x_{12}+x_{23}+\lambda_{13})-x_{34}=-\lambda_{13}+\lambda_{14}.
\end{array}
\right. 
$$

Finally one has 
$$g \diamond [\lambda_{13}, \lambda_{14}, \lambda_{24}] = [-\lambda_{13}+1, -\lambda_{13}+\lambda_{24}+1, -\lambda_{13}+\lambda_{14}].$$

Doing these computations for all the simple reflections we obtain the data in Table \ref{tableau action}:

\medskip

\begin{table}[h!]
\caption{Action of the simple reflections of $W(A_3)$ onto the set of admitted vectors.}
\begin{center}
\begin{tabular}{c|c|c|c|c|c|c}
                   &       [0,0,0]       &       [0,1,0]     &     [0,1,1]     &    [1,1,1]    &   [1,1,0]   &    [1,2,1]  \\
 \hline
 $s_{1,2}$     &      [1,1,0]        &       [1,1,1]     &     [1,2,1]     &    [0,1,0]    &   [0,0,0]   &    [0,1,1]    \\
 \hline
$s_{2,3}$      &   	 [1,1,1]	      &      [1,2,1]		&	   [1,1,0]     &    [0,0,0]    &   [0,1,1]   &    [0,1,0]   \\
\hline
$s_{3,4}$     &    	 [0,1,1]       &      [1,1,1]		&     [0,0,0]    &    [0,1,0]    &   [1,2,1]    &   [1,1,0]    \\
\end{tabular}
\end{center}
\label{tableau action}
\end{table}

\end{example}

\bigskip

\subsection{Computation of the $\diamond$ action}
The goal of this section is to give an explicit formula of the $\diamond$ action, that is we want to express the component $w \diamond X_{W(\widetilde{A}_n)}[\lambda]$ in terms of $X_{W(\widetilde{A}_n)}[\lambda]$. As mentioned in Section \ref{motivation}, we approach this question using the action on the admitted vectors instead of the components, in other words we want to express $w \diamond \lambda$ in terms of $\lambda$. In Proposition \ref{changement coordonnees} we give the formula of this action for $w=s_{k,\ell}$. There is no real difficulty to extend it for each element of $W$.

We keep the presentation of the root system $A_n$ as a triangle with base $\Delta$. Let $s_{k,\ell} \in W(A_n)$. Notice that the inversion set of $s_{k,\ell}$ is easy to express, indeed 
$$
 N(s_{k,\ell})=\{ e_k-e_{k+1}, e_k-e_{k+2},\cdots,e_k-e_{\ell}, e_{k+1}-e_{\ell}, e_{k+2}-e_{\ell},\cdots, e_{\ell-1}-e_{\ell} \}.
$$

Let $\lambda \in Y$ (the definition of $Y$ is given in Section \ref{notations}). We denote by $\lambda_{k,\ell}$ the coordinate in position $e_{k,\ell}$ of $\lambda$. See for example Figure \ref{pyra}.

We define now the following sets that will help us to understand how the action $\odot$ behaves.

\begin{definition}
Let $s_{k,\ell}$ be a reflection of $W(A_n)$. We denote 
 \begin{align*}
 A_{i,j}(k, \ell) &= \{ p \in \{i,\cdots,j-1\}~ | ~e_p-e_{p+1} \notin N(s_{k,\ell}) \},
\end{align*}
 $$B_{i,j}(k, \ell) = \{ p \in \{i,\cdots,j-1\}~ | ~ e_p-e_{p+1}\in N(s_{k,\ell}) \}.$$
\end{definition}

\begin{figure}[h!]
\centering
\includegraphics[scale=0.45]{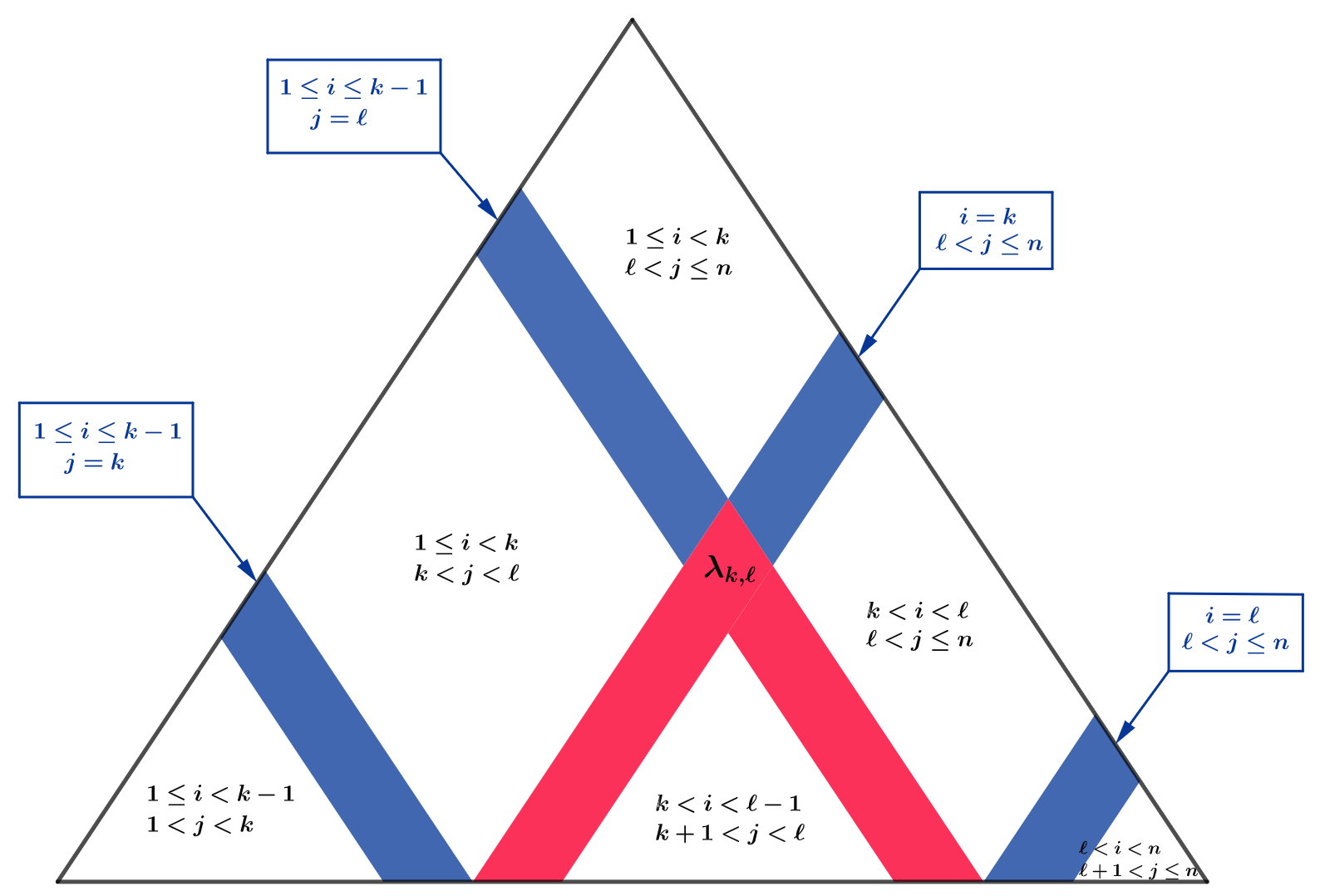} 
\caption{Those are the different regions of the root system $A_{n-1}$ cut out by the action of $s_{k,\ell}$. The red part is $N(s_{k,\ell})$. The blue part is the set of positive roots that are nontrivially permuted by $s_{k,\ell}$; the remaining positive roots are fixed.}
\label{pyra}
\end{figure}

\begin{definition}\label{integral}
 Let $w \in W(A_n)$. We define the function $\gamma_w$ by 
 $$
 \begin{array}{ccccc}
 \gamma_w & : & \widehat{X}_{W(\widetilde{A}_n)} \times [1,n+1] & \longrightarrow & \mathbb{Z} \\
				  &   &      (x,v)              & \longmapsto      & x_{w(\lfloor v \rfloor), w(\lfloor v+1 \rfloor)}.
 \end{array} 
 $$
 If $w=id$ we denote $\gamma_w$ by $\gamma$.
\end{definition}

\begin{lemma}\label{diff int}
Let $x \in X_{W(\widetilde{A}_n)}[\lambda]$, $t:=s_{k,\ell}$, $e_i - e_j \in \Phi^+$, and let us denote $y:=F(t)(x)$. Then we have the formula:
$$
\int\limits_{t(i)}^{t(j)} \gamma(x,v)dv - \int\limits_{i}^j\gamma(y,v)dv =-\int\limits_{i}^j\gamma_t(\lambda,v)dv + |B_{i,j}| .
$$
\end{lemma}

\begin{proof}
From the definition of $F(t)$ we have that 
$$
  y_{p,p+1}= \left\{
\begin{array}{ll}
 \sum\limits_{r=t(p)}^{t(p+1)-1}x_{r,r+1} + \lambda_{t(p),t(p+1)}~~\text{if}~~p \in A_{i,j} \\
 -\sum\limits_{r=t(p+1)}^{t(p)-1}x_{r,r+1} - \lambda_{t(p+1),t(p)}-1~~\text{if}~~ p \in B_{i,j}.
\end{array}
\right. 
$$
Therefore, using Definition \ref{integral} and the fact that $\lambda_{a,b}=-\lambda_{b,a}$, one can express $y_{p,p+1}$ as follows:
$$
  y_{p,p+1}= \left\{
\begin{array}{ll}
\displaystyle\int\limits_{t(p)}^{t(p+1)}\gamma(x,v)dv + \lambda_{t(p),t(p+1)}~~\text{if}~~p \in A_{i,j} \\
\displaystyle\int\limits_{t(p)}^{t(p+1)}\gamma(x,v)dv +\lambda_{t(p),t(p+1)}-1~~\text{if}~~ p \in B_{i,j}.
\end{array}
\right. 
$$
Consequently it follows that 
\begin{align}\label{diff int calcul}
&~~~~\int\limits_{t(i)}^{t(j)} \gamma(x,v)dv - \int\limits_{i}^j\gamma(y,v)dv \nonumber\\
&= \int\limits_{t(i)}^{t(j)} \gamma(x,v)dv - \sum\limits_{p=i}^{j-1}y_{p,p+1} \nonumber \\
& = \int\limits_{t(i)}^{t(j)} \gamma(x,v)dv - \sum\limits_{p \in A_{i,j}}(\int\limits_{t(p)}^{t(p+1)}\gamma(x,v)dv + \lambda_{t(p),t(p+1)}) - \sum\limits_{p \in B_{i,j}}(\int\limits_{t(p)}^{t(p+1)}\gamma(x,v)dv +\lambda_{t(p),t(p+1)}-1)\nonumber \\
& = \int\limits_{t(i)}^{t(j)} \gamma(x,v)dv - \sum\limits_{p \in A_{i,j}}\int\limits_{t(p)}^{t(p+1)}\gamma(x,v)dv - \sum\limits_{p\in A_{i,j}} \lambda_{t(p),t(p+1)}- \sum\limits_{p \in B_{i,j}}\int\limits_{t(p)}^{t(p+1)}\gamma(x,v)dv  \\ 
&~~-  \sum\limits_{p \in B_{i,j}}\lambda_{t(p),t(p+1)} + |B_{i,j}|\nonumber.
\end{align}

However, since $A_{i,j} \sqcup B_{i,j} = \{1,\dots ,j-1\}$ it is clear that
\begin{equation}\label{int vs sum}
 \int\limits_{t(i)}^{t(j)} \gamma(x,v)dv = \sum\limits_{p \in A_{i,j}}\int\limits_{t(p)}^{t(p+1)}\gamma(x,v)dv + \sum\limits_{p \in B_{i,j}}\int\limits_{t(p)}^{t(p+1)}\gamma(x,v)dv.
\end{equation}

Therefore, using (\ref{int vs sum}) in (\ref{diff int}) it follows that 
\begin{align*}
\int\limits_{t(i)}^{t(j)} \gamma(x,v)dv - \int\limits_{i}^j\gamma(y,v)dv & = -\sum\limits_{p\in A_{i,j}} \lambda_{t(p),t(p+1)}-  \sum\limits_{p \in B_{i,j}}\lambda_{t(p),t(p+1)} + |B_{i,j}| \\
& = - \sum\limits_{p=i}^{j-1}\lambda_{t(p),t(p+1)} + |B_{i,j}| \\
& = - \int\limits_{i}^j\gamma_t(\lambda, v)dv + |B_{i,j}|.
\end{align*}
\end{proof}

\begin{theorem} \label{changement coordonnees}
Let $\lambda$ be an admitted vector, $\alpha :=e_i - e_j$, $t:= s_{k,\ell}$, $A_{i,j} := A_{i,j}(k,\ell)$, $B_{i,j} := B_{i,j}(k,\ell)$. Then we have the formulas:
$$
(t \diamond \lambda)_{i,j} = \left\{
\begin{array}{ll}
\lambda_{t(i),t(j)} - \displaystyle\int\limits_{i}^j\gamma_t(\lambda, v)dv + |B_{i,j}|~~~~~~~~\text{if}~~\alpha \notin N(t) \\
\lambda_{t(i),t(j)} - \displaystyle\int\limits_{i}^j\gamma_t(\lambda, v)dv + |B_{i,j}|-1~~~\text{if}~~\alpha \in N(t).
\end{array}
\right.
$$
\end{theorem}

\begin{proof}
Let $x =(x_{i,j}) \in X_{W_a}[\lambda]$ with $\Theta(x)= \sum\limits_{i<j}x_{i,j}e_i \wedge e_j$, $\Theta(\lambda)= \sum\limits_{i<j}\lambda_{i,j}e_i \wedge e_j$, and $F(t)(x):=y$. Since $W_a$ acts on the components, it is enough to see where goes $x$ under this action. Let us denote by $\beta=(\beta_{i,j})$ the admitted vector such that $F(t)(x) \in X_{W(\widetilde{A}_n)}[\beta]$. The goal it then to express $\beta$ in terms of $\lambda$.

  This question is exactly the same as understanding the component $\Theta(X_{W(\widetilde{A}_n)}[\beta])$ in  $\bigwedge\nolimits^2(K)$. We will answer this question with the second point of view. We adopt the convention that for $j> i$, $x_{j,i} = -x_{i,j}$. These relations also apply to $\lambda$. 
  
  From Definition \ref{def action} and Corollary \ref{action sur reflexion} we have
\begin{align*}
\varphi_t ( \sum\limits_{ i < j } x_{i,j}e_i \wedge e_j) & = t \odot ( \sum\limits_{ i < j } x_{i,j}e_i \wedge e_j) \\
																		   & = \sum\limits_{ i < j } x_{i,j}e_{t(i)} \wedge e_{t(j)} - \sum\limits_{e_i-e_j \in N(t)}e_i \wedge e_j \\
																		   & =  \sum\limits_{ e_i-e_j \in N(t) } x_{i,j}e_{t(i)} \wedge e_{t(j)} +  \sum\limits_{ e_i-e_j \notin N(t) } x_{i,j}e_{t(i)} \wedge e_{t(j)} -  \sum\limits_{e_i-e_j \in N(t)}e_i \wedge e_j  \\
& = \sum\limits_{ e_i-e_j \in N(t) } x_{t(i),t(j)}e_i \wedge e_j +  \sum\limits_{ e_i-e_j \notin N(t) } x_{t(i),t(j)}e_i \wedge e_j -  \sum\limits_{e_i-e_j \in N(t)}e_i \wedge e_j \\
& =  \sum\limits_{ e_i-e_j \notin N(t) } x_{t(i),t(j)}e_i \wedge e_j + \sum\limits_{e_i-e_j \in N(t)} (x_{t(i),t(j)}-1) e_i \wedge e_j.
\end{align*}

Since $ \Theta \circ F(t) = \varphi_t \circ \Theta$ we obtain that

\begin{align*}
 & \text{~} \text{~}  \text{~} \text{~}  \Theta \circ F(t)(x) \\&= \sum\limits_{e_i-e_j \notin N(t)} x_{t(i),t(j)}  e_i \wedge e_j  + \sum\limits_{e_i-e_j \in N(t)} (x_{t(i),t(j)}-1) e_i \wedge e_j \\
&=  \sum\limits_{e_i-e_j \notin N(t)} ( \sum\limits_{r=t(i)}^{t(j)-1}x_{r,r+1}+\lambda_{t(i),t(j)} )    e_i \wedge e_j  -\sum\limits_{e_i-e_j \in N(t)} (\sum\limits_{r=t(j)}^{t(i)-1}x_{r,r+1}+\lambda_{t(j),t(i)}+1) e_i \wedge e_j \\
&=\sum\limits_{e_i-e_j \notin N(t)}(\int\limits_{t(i)}^{t(j)}\gamma(x,v)dv + \lambda_{t(i),t(j)})e_i \wedge e_j + \sum\limits_{e_i-e_j \in N(t)}(\int\limits_{t(i)}^{t(j)}\gamma(x,v)dv + \lambda_{t(i),t(j)}-1)e_i \wedge e_j.
\end{align*}

Moreover, we must have  
\begin{align*}
 \Theta \circ F(t)(x) & = \sum\limits_{e_i-e_j \notin N(t)} y_{i,j} e_i \wedge e_j  + \sum\limits_{e_i-e_j \in N(t)} y_{i,j} e_i \wedge e_j  \\
 & =  \sum\limits_{e_i-e_j \notin N(t)}( \sum\limits_{r=i}^{j-1}y_{r,r+1}+\beta_{i,j}) e_i \wedge e_j  + \sum\limits_{e_i-e_j \in N(t)} (\sum\limits_{r=i}^{j-1}y_{r,r+1}+\beta_{i,j}) e_i \wedge e_j \\
& = \sum\limits_{e_i-e_j \notin N(t)}( \int\limits_{i}^{j}\gamma(y,v)dv+ \beta_{i,j})e_i \wedge e_j + \sum\limits_{e_i-e_j \in N(t)}( \int\limits_{i}^{j}\gamma(y,v)dv+\beta_{i,j})e_i \wedge e_j.
\end{align*}

Hence we have
$$
\int\limits_{i}^{j}\gamma(y,v)dv+ \beta_{i,j} =
 \left\{
\begin{array}{ll}
\displaystyle\int\limits_{t(i)}^{t(j)}\gamma(x,v)dv + \lambda_{t(i),t(j)}~~~~~~~~\text{if}~~e_i-e_j \notin N(t) \\
\displaystyle\int\limits_{t(i)}^{t(j)}\gamma(x,v)dv + \lambda_{t(i),t(j)}-1~~~\text{if}~~e_i-e_j \in N(t),
\end{array}
\right.
$$
which is equivalent to
$$
 \beta_{i,j} =
 \left\{
\begin{array}{ll}
\lambda_{t(i),t(j)}+\displaystyle\int\limits_{t(i)}^{t(j)}\gamma(x,v)dv - \int\limits_{i}^{j}\gamma(y,v)dv ~~~~~~~~\text{if}~~e_i-e_j \notin N(t) \\
\lambda_{t(i),t(j)}+\displaystyle\int\limits_{t(i)}^{t(j)}\gamma(x,v)dv -\int\limits_{i}^{j}\gamma(y,v)dv -1~~~\text{if}~~e_i-e_j \in N(t).
\end{array}
\right.
$$

Moreover, because of Lemma \ref{diff int} we know that
$$
\int\limits_{t(i)}^{t(j)} \gamma(x,v)dv - \int\limits_{i}^j\gamma(y,v)dv =-\int\limits_{i}^j\gamma_t(\lambda,v)dv + |B_{i,j}| .
$$
Thus it follows that
$$
 \beta_{i,j} =
 \left\{
\begin{array}{ll}
\lambda_{t(i),t(j)}-\displaystyle\int\limits_{i}^j\gamma_t(\lambda,v)dv + |B_{i,j}| ~~~~~~~~\text{if}~~e_i-e_j \notin N(t) \\
\lambda_{t(i),t(j)}-\displaystyle\int\limits_{i}^j\gamma_t(\lambda,v)dv + |B_{i,j}| -1~~~\text{if}~~e_i-e_j \in N(t).
\end{array}
\right.
$$
This ends the proof since $ \beta_{i,j} = (t\diamond \lambda)_{i,j}$.
\end{proof}

\bigskip

\textbf{Acknowledgements}. We thank Christophe Hohlweg and Hugh Thomas for answering many questions and providing many helpful comments that help us to improve this paper. The author is also grateful to Christophe Reutenauer and Antoine Abram for valuable discussions.  We also thank the referees for useful suggestions. 

This work was partially supported by NSERC grants and by the LACIM.

\nocite{*}
\bibliographystyle{plain}
\bibliography{composante.bib}

\begin{thebibliography}{1}

\bibitem{abram2020order}
Antoine Abram, Nathan Chapelier-Laget, and Christophe Reutenauer.
\newblock An order on circular permutations.
\newblock {\em The {E}lectronic {J}ournal of {C}ombinatorics}, Volume 28,
  {I}ssue 3, 2021.

\bibitem{Brenti2005}
Anders Bj\"{o}rner and Francesco Brenti.
\newblock {\em Combinatorics of {C}oxeter groups}, volume 231 of {\em Graduate
  Texts in Mathematics}.
\newblock Springer, New York, 2005.

\bibitem{Bour68}
N.~Bourbaki.
\newblock {\em \'{E}l\'{e}ments de math\'{e}matique. {F}asc. {XXXIV}. {G}roupes
  et alg\`ebres de {L}ie. {C}hapitre {IV}: {G}roupes de {C}oxeter et syst\`emes
  de {T}its. {C}hapitre {V}: {G}roupes engendr\'{e}s par des r\'{e}flexions.
  {C}hapitre {VI}: syst\`emes de racines}.
\newblock Actualit\'{e}s Scientifiques et Industrielles, No. 1337. Hermann,
  Paris, 1968.

\bibitem{NC1}
Nathan Chapelier-Laget.
\newblock Shi variety corresponding to an affine {W}eyl group.
\newblock {\em arXiv preprint: 2010.04310}, 2020.

\bibitem{SRLE}
Matthew Dyer and Christophe Hohlweg.
\newblock Small roots, low elements, and the weak order in {C}oxeter groups.
\newblock {\em Adv. Math.}, 301:739--784, 2016.

\bibitem{GP}
Meinolf Geck and G\"{o}tz Pfeiffer.
\newblock {\em Characters of finite {C}oxeter groups and {I}wahori-{H}ecke
  algebras}, volume~21 of {\em London Mathematical Society Monographs. New
  Series}.
\newblock The Clarendon Press, Oxford University Press, New York, 2000.

\bibitem{Hu90}
James~E. Humphreys.
\newblock {\em Reflection groups and {C}oxeter groups}, volume~29 of {\em
  Cambridge {S}tudies in {A}dvanced {M}athematics}.
\newblock Cambridge University Press, Cambridge, 1990.

\bibitem{Kane2001}
Richard Kane.
\newblock {\em Reflection groups and invariant theory}, volume~5 of {\em CMS
  Books in Mathematics/Ouvrages de Math\'{e}matiques de la SMC}.
\newblock Springer-Verlag, New York, 2001.

\bibitem{JYS1}
Jian~Yi Shi.
\newblock Alcoves corresponding to an affine {W}eyl group.
\newblock {\em J. London Math. Soc. (2)}, 35(1):42--55, 1987.

\end{thebibliography}

\end{document}